\documentclass[a4paper,11pt]{article}
\usepackage[UKenglish]{babel}
\usepackage{authblk}

\usepackage[colorlinks=true]{hyperref}
\hypersetup{allcolors=blue!40!black}
\usepackage{color}
\usepackage{amsfonts,amsthm,amssymb}
\usepackage{mathtools}
\usepackage{dsfont}
\usepackage{float}
\usepackage{epsfig}
\usepackage{graphicx}
\usepackage{caption}
\usepackage{subcaption}
\usepackage{kotex}
\usepackage{tikz}

\newcommand{\SKT}{\mathsf{SKT}}
\newcommand{\DDS}{\mathsf{DDS}}
\usepackage{pgfplots}
\pgfplotsset{
    tick align=outside,
    x grid style={white},
    xmajorgrids,
    y grid style={white},
    ymajorgrids,
    axis line style={white},
    axis background/.style={fill=white!92!black},
    legend style={draw=white, fill=white},
    legend cell align={left}
}

\newtheorem{remark}{Remark}

\newtheorem{thm}{Theorem}[section]

\newtheorem{prop}[thm]{Proposition}

\numberwithin{equation}{section}
\numberwithin{thm}{section}
\numberwithin{problem}{section}
\usepackage[tight,english]{minitoc}
\newcommand{\beq}{\begin{equation}}
\newcommand{\eeq}{\end{equation}}
\newcommand {\f}{\dfrac}
\newcommand {\pa}{\partial}
\newcommand {\e}{\varepsilon}

\newcommand{\R}{\mathbb{R}}

\newcommand\N{{\mathbb N}}

\newcommand\MScN[1]{\href{http://www.ams.org/mathscinet-getitem?mr=#1}{\nolinkurl{(#1)}}}
\newcommand\DOI[1]{\href{http://dx.doi.org/#1}{(doi: \nolinkurl{#1})}}
\newcommand\LINK[1]{\href{#1}{(link: \nolinkurl{#1})}}
\newenvironment{acknowledgment}{\noindent{\bf Acknowledgments.}}{}

\title{Cross-diffusion and fast-reaction in pattern formation: a structural analysis
}
\author{E. Brocchieri\footnote{Department of Mathematics and Scientific Computing, University of Graz. Heinrichstr.~36, 8010 Graz, Austria. E-mail: \href{mail to: elisabetta.brocchieri@uni-graz.at}{elisabetta.brocchieri@uni-graz.at}}, 
C. Soresina\footnote{Department of Mathematics, University of Trento. Via Sommarive~14, 38123 Povo~(TN), Italy. E-mail: \href{mail to: cinzia.soresina@unitn.it}{cinzia.soresina@unitn.it}}}
\date{}
\providecommand{\keywords}[1]{\small\textit{{Keywords.}} #1\\}
\providecommand{\subjclass}[1]{\small\textit{{2010 Mathematics Subject Classification.}} #1}
\begin{document}
\maketitle

\setcounter{secttocdepth}{3}
\begin{abstract}
\noindent
Cross-diffusion systems play a central role in mathematical modelling, in which density-dependent dispersal and multiscale mechanisms can lead to spatial segregation and diffusion-driven instabilities. In several relevant examples, including generalised SKT-type competition models, cross-diffusion terms can be rigorously derived as fast-reaction limits, thereby providing a clear biological interpretation while posing significant analytical challenges.
In this work, we investigate the impact of biologically derived cross-diffusion on Turing instability. For a generalised SKT framework, we characterise instability conditions for a broad class of cross-diffusion functions arising from fast-reaction mechanisms. We then propose an alternative fast-reaction formulation leading to a different diffusion structure and show that, in this case, diffusion-driven pattern formation is prevented. We further discuss an example motivated by dietary diversity and starvation dynamics, and analyse how the sign structure of the reaction Jacobian interacts with cross-diffusion in determining the onset of patterns.
Our results clarify structural features that promote or inhibit spatial self-organisation in competitive systems.
\end{abstract}

\noindent
\section{Introduction}
Reaction–diffusion systems with cross-diffusion have attracted increasing attention in recent years, both for their analytical challenges and for their modelling relevance in biomathematics and pattern formation theory. Beyond the classical framework of diagonal diffusion, cross-diffusion mechanisms introduce nonlinear coupling at the level of spatial fluxes, leading to qualitatively new dynamical behaviours, including spatial segregation, coexistence mechanisms, and cross-diffusion-driven instabilities.

A benchmark example in this context is the Shigesada–Kawasaki–Teramoto (SKT) system for competing populations~\cite{SKT}. In this model, a cross-diffusion term in one species accounts for density-dependent dispersal driven by interspecific interactions (usually referred to as the triangular SKT model). From a modelling perspective, these terms allow for the coexistence of competing species through partial spatial segregation. From the analytical viewpoint, the SKT system has stimulated a vast literature addressing both rigorous well-posedness results and qualitative properties of solutions. Moreover, it has been extensively investigated in connection with pattern formation phenomena; we refer, for instance, to the general framework developed in~\cite{Breden2019, Gambino2012, Gambino2013, IMN}, as well as to the analytical contributions.

The strongly coupled structure of cross-diffusion systems makes the questions of existence and regularity of solutions particularly challenging, and their analysis requires advanced mathematical tools \cite{BoutonDesvillettesDietert2025, BrocchieriDesvillettes,  Desvillettes2015, DesvillettesRDM}.


An additional layer of interest comes from the derivation of cross-diffusion systems via fast-reaction limits~\cite{IMN}. In this approach, one starts from a larger system in which one species exhibits dichotomous states evolving on different time scales. By exploiting time-scale separation and state dichotomy, one rigorously derives a cross-diffusion term (triangular structure) in the singular limit, using entropy methods and compactness arguments.  The fast-reaction limit approach highlights the entropy structure of the fast-reaction system and rigorously yields the emergence of cross-diffusion terms. On the other hand, this method provides a structural and biological justification of the cross-terms, which encode processes occurring at faster temporal scales. This makes the fast-reaction limit approach a promising multi-scale derivation technique that bridges theory, modelling, and applications. Such derivations have been successfully employed in several contexts, including predator–prey dynamics \cite{ConfortoDesSoresina2018, DesvillettesFiorentinoMautone, DesvillettesSoresina, GambinoIGP, IIR}, autotoxicity in plant dynamics \cite{Morgan2026} and starvation mechanisms \cite{BCDK, Brocchieri2025}, models of population dynamics in historical settings (e.g., Neolithic transitions) \cite{Elias2018}, aggregation of cockroaches \cite{cockroaches}, corrosion processes on metal surfaces \cite{Corrosione}, and Michaelis–Menten type kinetics \cite{BaoNgocMM}. From a mathematical standpoint, proving convergence from the fast-reaction system to the effective cross-diffusion model is a delicate problem; from a modelling viewpoint, this approach yields the ``right'' cross-diffusion structure with a clear mechanistic interpretation.

A natural and still largely open question concerns the role of cross-diffusion in diffusion-driven instability. In some settings, such as the SKT model, autotoxicity models, or corrosion systems \cite{Corrosione, GianninoIuorioSoresina}, cross-diffusion may promote the emergence of stable spatial patterns. In other cases, for instance in certain predator–prey systems, it may instead inhibit pattern formation \cite{DesvillettesSoresina}. This raises a more general issue: which classes of biologically derived cross-diffusion terms are capable of inducing Turing instability, possibly even in systems that do not exhibit a classical activator–inhibitor structure? Importantly, this question should be addressed without losing the connection with cross-diffusion structures arising from fast-reaction limits. Recent contributions, such as~\cite{villar2025designing}, analyse pattern formation for reaction-diffusion systems with non-diagonal diffusion matrices in a rather general setting. However, in that framework, the diffusion structure is not necessarily tied to a specific biological derivation via time-scale separation. In the present work, we retain this link and focus on cross-diffusion terms arising from fast-reaction mechanisms.

More precisely, we start from a generalised SKT-type competition model and derive Turing instability conditions for a general cross-diffusion function within the class induced by fast-reaction limits. We then propose an alternative fast-reaction formulation that leads to a different cross-diffusion structure and show that, unlike the previous case, this variant does not support diffusion-driven pattern formation. 
Still within the competitive framework at the reaction level, we also discuss an example inspired by dietary diversity and starvation mechanisms, illustrating how the underlying biological assumptions influence the resulting cross-diffusion and its pattern-forming properties. Finally, we provide a systematic analysis of the sign structure of the reaction Jacobian and of the associated cross-diffusion terms (derived from fast-reaction limits), highlighting their interplay in determining the onset of Turing instability. 

To conclude, our analysis provides an overview of the structure of the cross-diffusion system by linking the reaction and cross-diffusion terms through the properties of the fast-reaction terms, thereby shedding light, on the one hand, on Turing instability and, on the other hand, on the entropy functional associated with the rigorous fast-reaction limit. Therefore, our results help clarify which biologically consistent cross-diffusion mechanisms can generate spatial self-organisation and which structural features, instead, prevent it.


\section{Problem setting}
In this section, we introduce the SKT model, which describes the competition between two species, with one species tending to avoid the other. A cross-diffusion term models this, and we briefly recall its derivation in the fast-reaction limit. We then propose an alternative fast-reaction framework that yields a different cross-diffusion model, distinct from the classical SKT structure.

\subsection{The SKT model: competition and avoidance}
We recall here the general form of the SKT model for competing species. Denoting by $u=u(t,x)$ and $v=v(t,x)$ the densities at time $t$ and position $x$ of two different species competing for the same resource on a bounded domain $\Omega\subset\mathbb{R}^\kappa,\, \kappa\in \mathbb{N}$, the system writes
\begin{equation}\label{SKT model +}
\begin{cases}	
\partial_{t}u-\Delta\left(d_{u}u+d_{12}\,\phi(v)u\right)=u(r_{u}-r_{11}u-r_{12}v),&(t,x)\in\mathbb{R}_{+}\times\,\Omega,\\
\partial_{t}v-d_{v}\Delta v=v(r_{v}-r_{21}u-r_{22}v),
&(t,x)\in\mathbb{R}_{+}\times\,\Omega,\\
u(0,x)=u_0(x),\, v(0,x)=v_0(x), & x\in \Omega,\\
\nabla u\cdot \sigma=\nabla v\cdot \sigma=0,
&(t,x)\in\mathbb{R}_{+}\times\,\partial\Omega,
\end{cases}
\end{equation}
where $\sigma$ indicates the unit normal vector to $\partial \Omega$. 
The diffusion term in the equation for $u$ has a standard diffusion term with a diffusion coefficient $d_u>0$, and a cross-diffusion term depending on $v$ with a cross-diffusion coefficient $d_{12}>0$ and the cross-diffusion function $\phi(v)$ being positive and non-decreasing. On the contrary, the movements of species $v$ are not affected by $u$, so that only the standard diffusion term appears in the equation. The reaction part describes the growth and competition of the two species, with maximum growth rates $r_u,\, r_v$, intra-specific competition rates $r_{11}, r_{22}$, and inter-specific competition rates $r_{12}, r_{21}$, respectively, for species $u$ and $v$.

We define the diffusivity function as
\begin{equation}\label{def D_+}
    D_+(u,v)\coloneqq u(d_u+d_{12}\phi(v))>0,
\end{equation}
where $D_+$ remarks the fact that there is an increase in the diffusion coefficient due to the presence of population $v$.

As shown in~\cite{Desvillettes2015, IMN}, system \eqref{SKT model +} can be obtained by taking the limit as $\e\to 0$ of the following fast-reaction system 
\begin{equation}\label{meso system +}
\begin{cases}
\pa_t u_a-d_u\Delta u_a=u_a(r_{u}-r_{11}u-r_{12}v)+\f 1{\e}\left[k(v)u_b-h(v)u_a\right],\\
\pa_t u_b-(d_u+d_{12})\Delta u_b=u_b(r_{u}-r_{11}u-r_{12}v)-\f 1{\e}\left[k(v)u_b-h(v)u_a\right],\\
\pa_t v-d_v\Delta v=v(r_{v}-r_{21}u-r_{22}v).
\end{cases}
\end{equation}
In particular, species $u$ is divided into two states, $u_a$ (quiet, because $v$ is not abundant) and $u_b$ (excited, because $v$ is present). The switch between these two states is regulated by the $v$-dependent transition rates $h$ and $k$. For their biological meaning,~$h$ is assumed non-decreasing, while $k$ is non-increasing, namely for all $v\geq 0$
\begin{equation}\label{hyp_hk}
 0<h(v), \, k(v)\leq 1, \quad h'(v)\geq 0, \quad k'(v)\leq 0.
\end{equation} 
Moreover, the excited state disperses more than the quiet state. This is modelled by assigning a higher diffusion coefficient to $u_b$ than to $u_a$, namely $d_u$ for state $u_a$, and $d_u+d_{12}$ for state $u_b$.

At a formal level, we can see the link between the fast-reaction system~\eqref{meso system +} and the limiting cross-diffusion system. Formally adding the first two equations of system~\eqref{meso system +}, we obtain the equation for the total population $u=u_a+u_b$ and take the limit as $\e\to 0$, thus obtaining the limiting system 
\begin{equation}\label{limit system +}
    \begin{cases}
        \pa_t u-\Delta\left(
        d_u u +d_{12} u_b\right)=u(r_{u}-r_{11}u-r_{12}v),\\
        \pa_t v-d_v\Delta v=v(r_{v}-r_{21}u-r_{22}v),
    \end{cases}
\end{equation}
where
\begin{equation}\label{limit Q=0 + and -}
    \begin{cases}
        u=u_a+u_b,\\
        Q_{\SKT}:=k(v)u_b-h(v)u_a=0,
    \end{cases}
    \quad \Longrightarrow \quad u_b=u_b(u,v)=\f{h(v)}{h(v)+k(v)} u.
\end{equation}
Therefore, the fast-reaction system~\eqref{meso system +} (formally) converges to the cross-diffusion system~\eqref{SKT model +} where
\begin{equation}\label{condition limit +}
    \phi(v)=\f{h(v)}{h(v)+k(v)}. 
\end{equation}
Thanks to the hypothesis on the transition rates~\eqref{hyp_hk}, we also have
\begin{equation}\label{phi'}
    \phi'(v)=\f{h'(v)k(v)-h(v)k'(v)}{(h(v)+k(v))^2}\geq0.
\end{equation}
Thus, it yields 
\begin{equation}\label{grad A+}
    \pa_{1}D_{+}(u,v)=d_u+d_{12}\phi(v)\ge d_u>0,\quad\qquad\pa_{2}D_+(u,v)=d_{12}u\phi'(v)\geq 0.
\end{equation}
\begin{remark}
By choosing
\begin{equation}\label{ex h,k +}
    h(v)= \f v M,\qquad k(v)=1-\f v M,
\end{equation}
where the constant $M$ is strictly related to the maximum principle applied to the equation of $v$,
we obtain the cross-diffusion system \eqref{SKT model +} with the increasing function $\phi(v)= v/M$ proposed by Shigesada, Kawasaki, and Teramoto~\cite{SKT}, which is considered as a benchmark problem for cross-diffusion systems in population dynamics. 
\end{remark}

\begin{remark}
The rigorous derivation of the cross-diffusion system \eqref{limit system +} is obtained by the fast-reaction limit. Entropy and duality methods are used to get the necessary compactness and pass to the limit as $\varepsilon \to 0$ \cite{Desvillettes2015}.
\end{remark}

\subsection{The SKT model: competition and hiding behaviour}
Following the same intuition as in the previous section, we consider a fast-reaction system in which species $u$ is still divided into two states, $u_a$ (quiet) and $u_b$ (excited). Still, the excited state exhibits a different dispersal behaviour. In particular, we assume that excited individuals tend to hide in the presence of population $v$. This is modelled by assigning a smaller diffusion coefficient to $u_b$ than to $u_a$, namely~$d_u$ for state $u_a$, and $d_u-d_{12}$ for state $u_b$, assuming that $d_{12} < d_u$.
Under these assumptions, the fast-reaction system then writes
\begin{equation}\label{meso system -}
\begin{cases}
\pa_t u_a-d_u\Delta u_a=u_a(r_{u}-r_{11}u-r_{12}v)+\f 1{\e}\left[k(v)u_b-h(v)u_a\right],\\
\pa_t u_b-(d_u-d_{12})\Delta u_b=u_b(r_{u}-r_{11}u-r_{12}v)-\f 1{\e}\left[k(v)u_b-h(v)u_a\right],\\
\pa_t v-d_v\Delta v=v(r_{v}-r_{21}u-r_{22}v),
\end{cases}
\end{equation}
where $u=u_a+u_b$, and parameters and transition rates have the same biological meaning as in the previous case. In particular, $h$ and $k$ satisfy \eqref{hyp_hk} ($h(v)$ is nondecreasing, whereas $k(v)$ is nonincreasing).

From the fast-reaction system~\eqref{meso system -}, we (formally) write the equation satisfied by the total population~$u=u_a+u_b$ and take the limit as $\e\to 0$, to get the limiting cross-diffusion system
\begin{equation*}
    \begin{cases}
        \pa_t u-\Delta\left(
        d_u u -d_{12} u_b\right)=u(r_{u}-r_{11}u-r_{12}v),\\
        \pa_t v-d_v\Delta v=v(r_{v}-r_{21}u-r_{22}v),
    \end{cases}
\end{equation*}
where, as in the avoidance case (see~\eqref{limit Q=0 + and -}), state $u_b$ is given by
$$u_b=u_b(u,v)=\f{h(v)}{h(v)+k(v)} u.$$
As in the previous section (see equation~\eqref{condition limit +} and~\eqref{phi'}), by defining 
\begin{equation}\label{identity SKT -}
   \phi(v):=\f{h(v)}{h(v)+k(v)}, \qquad \phi'(v)\geq 0,
\end{equation}
the fast-reaction system \eqref{meso system +} (formally) converges to the cross-diffusion system 
\begin{equation}\label{SKT model -}
\begin{cases}	
\partial_{t}u-\Delta\left(d_{u}u-d_{12}u\,\phi(v)\right)=u(r_{u}-r_{11}u-r_{12}v),\hspace{1cm}&\hspace{-1cm}(t,x)\in\mathbb{R}_{+}\times\,\Omega,\\
\partial_{t}v-d_{v}\Delta v=v(r_{v}-r_{21}u-r_{22}v),
&\hspace{-1cm}(t,x)\in\mathbb{R}_{+}\times\,\Omega,\\
\nabla u\cdot \sigma=\nabla v\cdot \sigma=0\,,
&\hspace{-1cm}(t,x)\in\mathbb{R}_{+}\times\,\partial\Omega,
\end{cases}
\end{equation}
with the diffusivity function 
\begin{equation}\label{def D_-}
    D_-(u,v)\coloneqq u(d_u-d_{12}\phi(v)).
\end{equation}
Recalling \eqref{identity SKT -} and the assumption $d_{12}<d_u$ on the diffusion coefficients, we have the strict positivity of $D_-$. Regarding its derivatives, we have
\begin{equation}\label{grad A-}
\pa_{1}D_-(u,v)=d_u-d_{12}\phi(v)>0,\qquad
\pa_{2}D_-(u,v)=-d_{12}u\phi'(v)\leq 0.
\end{equation} 
Therefore, we consider the general class of models with diffusivity function as in equation~\eqref{def D_-} where the function $\phi(v)$ is positive, bounded by $1$, and non-decreasing.
\begin{remark}
The same cross-diffusion limit arises if population $u_a$ is assumed to be less motile than $u_b$, provided the monotonicity of the transition rates is reversed, i.e., $h(v)$ is non-increasing, and $k(v)$ is non-decreasing. From a biological perspective, this might correspond to a territorial defence mechanism, where $u_a$ denotes a fighting (aggressive) state against $v$, whereas $u_b$ represents a quiet state. Despite the different biological interpretation, the resulting cross-diffusion structure in the fast-reaction limit remains unchanged.
\end{remark}
As for the SKT model, by choosing the transition rates as in \eqref{ex h,k +}, namely
\begin{equation*}
    h(v)= \f v M,\qquad k(v)=1-\f v M, 
\end{equation*}
where the constant $M$ is strictly related to the maximum principle applied to the equation of $v$,
we obtain the cross-diffusion system \eqref{SKT model -} with $\phi(v)=v /M$, corresponding to a decreasing diffusivity function $D_-(u,v)$. 

\begin{remark}
The rigorous derivation of the cross-diffusion system \eqref{SKT model -} is obtained by the fast-reaction limit, similarly to the avoidance case \eqref{SKT model +} by entropy and duality methods. A key tool to rigorously pass to the limit $\e\to 0$ is the analysis of the same entropy functional introduced in \cite{Desvillettes2015}.
\end{remark}
\section{Linear stability analysis}
In this section, we perform the linear stability analysis of the homogeneous steady states of the systems \eqref{SKT model +} and \eqref{SKT model -}. In both cases, we aim to establish conditions that lead to cross-diffusion-induced instability and pattern formation. While the stability properties of homogeneous equilibria with respect to homogeneous perturbations are the same in both cases (see Proposition \ref{ode stability prop}), the conditions for heterogeneous perturbations are different (see Theorem \ref{Turing thm avoidance} and \ref{Turing thm aggregation}).

For the sake of clarity, we define the following quantities, which depend only on the reaction coefficients
\begin{equation}\label{def R,S,T}
    R\coloneqq r_{11}r_{22}-r_{12}r_{21},\hspace{0,8cm}S\coloneqq r_v\,r_{11}-r_u\,r_{21},\hspace{0,8cm}T\coloneqq r_u\,r_{22}-r_v\,r_{12}.
\end{equation}
Moreover, we distinguish two regimes in which a coexistence equilibrium for the reaction parts exists (positive). In the \textit{weak competition regime}, interspecific competition is weaker than the intraspecific one, and we have
\begin{equation}\label{weak regime}
\dfrac{r_{12}}{r_{22}}<\dfrac{r_u}{r_v}<\dfrac{r_{11}}{r_{21}}\qquad\Longrightarrow\qquad R,S,T>0.
\end{equation}
On the other hand, we have the \textit{strong competition regime}, in which interspecific competition is stronger than intraspecific competition. In this regime, the parameters satisfy
\begin{equation}\label{strong regime}
\dfrac{r_{11}}{r_{21}}<\dfrac{r_u}{r_v}<\dfrac{r_{12}}{r_{22}}\qquad\Longrightarrow\qquad R,S,T<0.
\end{equation}
\subsection{Equilibria of the reaction part}\label{subsect Turing ODE}
Since both systems \eqref{SKT model +} and \eqref{SKT model -} share the same reaction part, they both admit the following homogeneous steady states
\begin{equation*}
(u_1,\,v_1)\coloneqq(0,0),\hspace{1cm}(u_2,\,v_2)\coloneqq\Big(\dfrac{r_u}{r_{11}},0\Big),\hspace{1cm}(u_3,\,v_3)\coloneqq\Big(0,\dfrac{r_v}{r_{22}}\Big),
\end{equation*}
where $(u_1,v_1)$ describes the total extinction, the semi-trivial equilibria $(u_2,v_2)$ and $(u_3,v_3)$ represent the extinction of one species. Both systems also admit a steady state 
\begin{equation}\label{homo equi coexistence}
(u^*,\,v^*)\coloneqq\left(\dfrac{T}{R},\dfrac{S}{R}\right),
\end{equation}
being the quantities $R,\, S,\, T$ defined in~\eqref{def R,S,T}.
We point out that the positivity of both components of the coexistence equilibrium~\eqref{homo equi coexistence} is guaranteed in the weak regime \eqref{weak regime} as well as in the strong regime \eqref{strong regime}. 
\begin{prop}\textit{[Stability properties of equilibria of the reaction part]}\label{ode stability prop}\hfill\\
Considering the reaction part of systems \eqref{SKT model +} and \eqref{SKT model -}, we have the following stability properties.
\begin{itemize}
    \item [(i)] The trivial equilibrium $(u_1,v_1)$ is unstable in both the weak competition regime~\eqref{weak regime} and the strong competition regime~\eqref{strong regime}.
    \item [(ii)] The semi-trivial equilibria $(u_i,v_i)_{i=2,3}$ are unstable in the weak competition regime~\eqref{weak regime} and locally asymptotically stable in the strong competition regime~\eqref{strong regime}.
    \item[(iii)] The coexistence equilibrium $(u^*,v^*)$ is locally asymptotically stable in the weak competition regime~\eqref{weak regime} and unstable in the strong competition regime~\eqref{strong regime}.
\end{itemize}
\end{prop}

\noindent
\begin{proof}
By linearisation, the Jacobian matrix writes as
\begin{equation}\label{def Jacobian kinetics}
J(u, v)\coloneqq 
\begin{pmatrix}
r_u- 2r_{11} u-r_{12}v & -r_{12}u\\
-r_{21}v & r_v-r_{21}u-2r_{22}v 
\end{pmatrix},
\end{equation}
implying that the trivial equilibrium $(u_1,v_1)$ is linearly unstable for the associated ODE systems \eqref{SKT model +} and \eqref{SKT model -}, under both regimes \eqref{weak regime} and \eqref{strong regime}. Evaluating the Jacobian matrix~\eqref{def Jacobian kinetics} at the semi-trivial equilibria, we obtain
\begin{equation*}
   J(u_2,v_2)=
\begin{pmatrix}
-r_u & -{r_ur_{12}}/{r_{11}}\\[1,5ex]
0 & S/{r_{11}}
\end{pmatrix},\qquad
J(u_3,v_3)=
\begin{pmatrix}
T/{r_{22}} & 0\\[1,5ex]
-{r_vr_{21}}/{r_{22}} & -r_v
\end{pmatrix},
\end{equation*}
thus, item $(ii)$ is proved. 

We conclude by analysing the stability of the coexistence equilibrium $(u^*,v^*)$. The Jacobian matrix \eqref{def Jacobian kinetics} becomes
\begin{equation}\label{def Jacobian coex equi}
J^*\coloneqq J(u^*, v^*)=
\begin{pmatrix}
- r_{11} u^* & -r_{12}u^*\\
-r_{21}v^* & -r_{22}v^* 
\end{pmatrix},
\end{equation}
and gives
\begin{equation}\label{tr and det coex equi}
    \text{tr} J^*<0 \qquad\text{and}\qquad\det J^*=u^*v^*R=\f{ST}R.
\end{equation}
Therefore, for the reaction part of both systems~\eqref{SKT model +} and \eqref{SKT model -}, the steady state $(u^*, v^*)$ is locally asymptotically stable to homogeneous perturbations if and only if the parameters of the reaction part are in the weak competition regime \eqref{weak regime}.
\end{proof}

\subsection{Cross-diffusion-driven instability with avoidance behaviour}
In this section, we analyse the stability properties of the homogeneous equilibria 
of the reaction cross-diffusion system with the avoidance effect~\eqref{SKT model +}, where the diffusivity function takes the general form as in~\eqref{def D_+}. More specifically, we focus on the analysis of the coexistence equilibrium $(u^*, v^*)$ in~\eqref{homo equi coexistence} that is the most meaningful one from a biological viewpoint. Thus, we provide the conditions under which Turing instability can occur. 
\begin{thm}\textit{[Turing instability with avoidance behaviour]}\label{Turing thm avoidance}\hfill\\
We consider system~\eqref{SKT model +} with a diffusivity function in the general form~\eqref{def D_+} and we define the quantities 
\begin{equation}\label{alphabeta}
    \alpha\coloneqq r_{21}u^*\phi'(v^*)-r_{22}\phi(v^*),\qquad \beta\coloneqq d_ur_{22}v^*+d_vr_{11}u^*>0.
\end{equation}
Then, the coexistence equilibrium $(u^*,v^*)$ is linearly unstable to heterogeneous perturbations under the weak competition regime \eqref{weak regime}, only if $\alpha>0$ 
and if the cross-diffusion coefficient satisfies 
    \begin{equation}\label{ineq d12}
    d_{12}>\f{\beta }{\alpha v^*}
    +\f2{(\alpha v^*)^2}\left(d_v\phi(v^*)\det J^*+\sqrt{\Delta^{**}}\right),
    \end{equation}
    with
    \begin{equation}\label{deltad12}
    \Delta^{**}\coloneqq (d_v\phi(v^*)\det J^*)^2+\alpha v^*(\beta \phi(v^*)+d_v\alpha v^*)>0.
    \end{equation}
\end{thm}

\begin{remark}
We observe that when $\phi(v)=v$ (leading to the triangular SKT model with avoidance behaviour), the necessary condition $\alpha>0$ is satisfied and reduces to the one found in~\cite{Breden2019}, i.e.,
\[
    \f{r_u}{r_v}<\f 12\left(\f{r_{11}}{r_{21}}+\f{r_{12}}{r_{22}}\right),
    \]
as well as the condition~\eqref{ineq d12} in the particular case $d_u=d_v$.
\end{remark}

\begin{proof}
By linearising system \eqref{SKT model +} with diffusivity function~\eqref{def D_+} about a general steady state $(\bar u,\bar v)$ and using the spectral decomposition of $-\Delta$ with Neumann boundary condition, we obtain
\begin{equation}\label{def matrix Mn+}
M_{n,+}(\bar u, \bar v)\coloneqq 
\begin{pmatrix}
J_{11}-\pa_1D_+(\bar u, \bar v)\lambda_n & 
J_{12}-\pa_2D_+(\bar u, \bar v)\lambda_n\\
J_{21} & 
J_{22}-d_v\lambda_n
\end{pmatrix},
\end{equation}
where $(J_{ij})_{i,j=1,2}$ are the elements of the Jacobian matrix~\eqref{def Jacobian kinetics} at the equilibrium $(\bar u,\bar v)$, and $\lambda_n,\, n\in\N$ are the eigenvalues of the Laplace operator in the Neumann case.

%
Evaluating $M_{n,+}$ at the coexistence equilibrium $(u^*, v^*)$, the characteristic matrix reduces to 
\begin{equation*}
M_{n,+}(u^*, v^*)\coloneqq 
\begin{pmatrix}
-r_{11}u^{*}-\pa_1D_+(u^*, v^*)\lambda_n\quad & 
-r_{12}u^{*}-\pa_2D_+(u^*, v^*)\lambda_n\\
-r_{21}v^{*} & -r_{22}v^{*}-d_v\lambda_n
\end{pmatrix}.
\end{equation*}
Since the trace is negative (see \eqref{grad A+}), Turing instability can occur only if 
\begin{equation}\label{det coex equi <0}
    \det M_{n,+}(u^*, v^*)=d_v\pa_1D_+(u^*, v^*)\lambda_n^2+B\lambda_n+\det J<0,
\end{equation}
where
\begin{equation}\label{def B coex equi}
    B=\pa_1D_+(u^*, v^*)r_{22}v^*+d_vr_{11}u^*-\pa_2D_+(u^*, v^*)r_{21} v^*.
\end{equation}
Recalling the positivity of $\pa_1D_+$ from \eqref{def D_+} and of $\det J^*$ (see \eqref{weak regime} and \eqref{tr and det coex equi}), the condition \eqref{det coex equi <0} is satisfied only if $B<0$. Therefore, Turing instability can occur only if the conditions below are fulfilled
\begin{equation}\label{system per neces cond}
    \begin{cases}
        B<0,\\
        \Delta^*\coloneqq B^2-4d_v\pa_1D_+(u^*, v^*)\det J^*>0.
    \end{cases}
\end{equation}
We now prove that the condition~\eqref{system per neces cond} is equivalent to 
\begin{equation}\label{system conditions d12}
\begin{cases}
    d_{12}>\tilde{d}_{12}=\f{\beta }{\alpha v^*},\\
    d_{12}<d_{12_-}\quad \vee\quad d_{12}>d_{12_+},
\end{cases}
\end{equation}
where $d_{12_-}$ and $d_{12_+}$ correspond to the real roots of the polynomial in the second condition of \eqref{system conditions d12} and are defined as
\begin{equation}\label{def roots d12}
d_{12_{\pm}}\coloneqq \f{\beta }{\alpha v^*}+\f2{(\alpha v^*)^2}\left(d_v\phi(v^*)\det J^*\pm\sqrt{\Delta^{**}}\right),
\end{equation}
where $\Delta^{**}$ is defined in~\eqref{deltad12}.
Using the definitions in \eqref{grad A+}, \eqref{def B coex equi}, the first condition in \eqref{system per neces cond} rewrites as
\begin{equation}\label{equi cond B<0}
B=d_ur_{22}v^*+d_vr_{11}u^*+d_{12}v^*\left(r_{22}\phi(v^*)-r_{21}u^*\phi'(v^*)\right)=\beta-d_{12}v^*\alpha<0.
\end{equation}
Therefore, the inequality above holds only if $\alpha>0$, thus giving the first condition in \eqref{system conditions d12}. 

Using~\eqref{grad A+},~\eqref{alphabeta} and~\eqref{equi cond B<0}, the second condition in \eqref{system per neces cond} rewrites as
\begin{align}\label{para d12>0}
\Delta^*&=(\beta-d_{12}v^*\alpha)^2-4d_v(d_{u}+d_{12}\phi(v^*))\det J^*\notag\\
&=(\alpha v^*)^2d^2_{12}-2\left(\alpha\beta v^*+2d_v\phi(v^*)\det J^*\right)d_{12}+\beta^2-4d_ud_v\det J^* >0.
\end{align}
We first observe that the coefficient of the linear term in \eqref{para d12>0} is strictly negative and that the non homogeneous coefficient $\beta^2-4d_ud_v\det J^*$ is strictly positive (remember that $\det J^* = u^*v^*(r_{11}r_{22}-r_{12}r_{21})>0$). Therefore, to determine the existence of real solutions of the parabola in \eqref{para d12>0}, we compute its discriminant 
\begin{align*}
    \Delta^{**}&=(\alpha\beta v^*+2d_v\phi(v^*)\det J^*)^2-(\alpha\beta v^*)^2+4d_ud_v\alpha^2(v^*)^2\det J^*\\
    &=4d_v^2(\phi(v^*))^2(\det J)^2+4d_v\alpha\beta v^*\phi(v^*)\det J^*+4d_ud_v(\alpha v^*)^2\det J^*>0,
\end{align*}
thanks to $\alpha, \det J^*>0$. Thus, there exist two real, distinct, and positive roots $d_{12_-}<d_{12_+}$ (remember that $-2\left(\alpha\beta v^*+2d_v\phi(v^*)\det J^*\right)<0$) so that the second condition in \eqref{system conditions d12} and \eqref{def roots d12} hold.

From \eqref{system conditions d12} we observe that $d_{12_-}<\tilde d_{12}<d_{12_+}$ with $d_{12_\pm}$ as in \eqref{def roots d12}, thus giving the result.
\end{proof}

\subsection{Cross-diffusion-driven instability with hiding behaviour}
In this section, we analyse the stability properties of the homogeneous equilibria
of the reaction cross-diffusion system with hiding behaviour \eqref{SKT model -}. More specifically, we focus
on the analysis of the coexistence equilibrium $(u^*,v^*)$ in~\eqref{homo equi coexistence} that is the most
meaningful one from a biological viewpoint. Thus, we provide the conditions under
which Turing instability can occur.
\begin{thm}\textit{[Linearised analysis with hiding behaviour]}\label{Turing thm aggregation}\hfill\\
We consider system~\eqref{SKT model -} with a diffusivity function in the general form~\eqref{def D_-}. Then, the coexistence equilibrium $(u^*,v^*)$ is locally stable to heterogeneous perturbations under the weak competition regime \eqref{weak regime}. 
\end{thm}
Therefore, in the weak competition regime, cross-diffusion does not destabilise the homogeneous equilibrium, and therefore, Turing patterns do not occur, unlike the case of avoidance behaviour.

\begin{proof}
Similarly to~\eqref{def matrix Mn+}, we linearise system \eqref{SKT model -} with diffusivity function \eqref{def D_-} about a general steady state $(\bar u,\bar v)$ and we get 
\begin{equation*}
M_{n,-}(\bar u, \bar v)\coloneqq 
\begin{pmatrix}
J_{11}-\pa_1D_-(\bar u,\bar v)\lambda_n & J_{12}-\pa_2D_-(\bar u,\bar v)\lambda_n\\
J_{21} & J_{22}-d_v\lambda_n
\end{pmatrix},
\end{equation*}
where $(J_{ij})_{i,j=1,2}$ are the elements of the Jacobian matrix evaluated at the equilibrium $(\bar u,\bar v)$, and $\lambda_n,\, n\in\N$ are the eigenvalues of the Laplace operator in the Neumann case.

Evaluating $M_{n,-}(u^{*},v^{*})$ at the coexistence equilibrium $(u^*, v^*)$, we observe that it has a negative trace (remember~\eqref{def Jacobian coex equi}).
Therefore, Turing instability can occur only if 
\begin{equation*}
    \det M_{n,-}(u^*, v^*)=d_v\pa_1D_-(u^{*},v^{*})\lambda_n^2+B\lambda_n+\det J^*<0,
\end{equation*}
with 
\begin{equation*}
    B=\pa_1D_-(u^{*},v^{*})r_{22}v^*+d_vr_{11}u^*-\pa_2D_-(u^{*},v^{*})r_{21} v^*.
\end{equation*}
Recalling the positivity of $\pa_1D_-$ and $\det J^*$ in~\eqref{grad A-} and~ \eqref{weak regime},~\eqref{tr and det coex equi}, respectively, then condition \eqref{det coex equi <0 aggregation} is satisfied only if $B<0$. However, since from~\eqref{grad A-} we have $\pa_1D_->0$ and $\pa_2D_-<0$, it cannot be $B<0$, independently of the explicit form of the function~$\phi$.
\end{proof}



\section{Sign-structure classification}
In this section, we classify triangular cross-diffusion systems according to the sign structure of the Jacobian matrix of the reaction part evaluated at a coexistence equilibrium, with particular emphasis on the possibility of cross-diffusion-driven instability.

We consider a general reaction--diffusion system on a bounded domain $\Omega$ describing the dynamics of two species $u$ and $v$, in which the first species exhibits cross-diffusion effects, while the second one undergoes standard diffusion. A general form is
\begin{equation}\label{generalXDmodel}
\begin{cases}	
\partial_{t}u-\Delta\left(D(u,v)\right)=f(u,v),&(t,x)\in\mathbb{R}_{+}\times\,\Omega,\\
\partial_{t}v-d_{v}\Delta v=g(u,v),
&(t,x)\in\mathbb{R}_{+}\times\,\Omega,\\
u(0,x)=u_0(x),\, v(0,x)=v_0(x), & x\in \Omega,\\
\nabla u\cdot \sigma=\nabla v\cdot \sigma=0,
&(t,x)\in\mathbb{R}_{+}\times\,\partial\Omega,
\end{cases}
\end{equation}
where the diffusivity function $D(u,v)$ satisfies 
\begin{equation}
    D(u,v)>0,\quad\textnormal{and}\quad \pa_1D(u,v)>0, \quad \forall\,u,\,v>0.
\end{equation}
The reaction part is general and indicated by the functions $f$ and $g$. We also consider initial conditions $u_0$ and $v_0$, and homogeneous Neumann boundary conditions.

Let $(u^*,v^*)$ be a coexistence equilibrium of the reaction part, and denote by
\[
J^* =
\begin{pmatrix}
f_u & f_v \\
g_u & g_v
\end{pmatrix}_{(u^*,v^*)}
\]
the Jacobian matrix of the reaction terms evaluated at $(u^*,v^*)$.

We are concerned with the possible emergence of Turing patterns and the role of the cross-diffusion term. The goal is to relate the sign structure of the Jacobian matrix with the diffusivity function (in particular, the sign of $\pa_2D$) and its fast-reaction derivation.
Since each entry of $J^*$ may be either positive or negative, there are $2^4 = 16$ possible sign configurations. Among these, eight necessarily correspond to an unstable equilibrium of the reaction system and are therefore excluded. The remaining eight configurations are those that may satisfy the stability conditions
\[
\operatorname{tr} J^* = f_u + g_v < 0,
\qquad
\det J^* = f_u g_v - f_v g_u > 0.
\]

Among these configurations, four correspond to the classical activator--inhibitor structure (i.e., $f_v g_u < 0$), which is known to support diffusion-driven instability even in the absence of cross-diffusion. In particular, they are
\[
\begin{pmatrix}
+ & +\\
- & -
\end{pmatrix},\quad
\begin{pmatrix}
+ & -\\
+ & -
\end{pmatrix},\quad
\begin{pmatrix}
- & +\\
- & +
\end{pmatrix},\quad
\begin{pmatrix}
- & -\\
+ & +
\end{pmatrix}.
\]
The remaining four 
\[
\begin{pmatrix}
- & -\\
- & -
\end{pmatrix},\quad
\begin{pmatrix}
- & +\\
- & -
\end{pmatrix},\quad
\begin{pmatrix}
- & -\\
+ & -
\end{pmatrix},\quad
\begin{pmatrix}
- & +\\
+ & -
\end{pmatrix}.
\]
do not exhibit an activator--inhibitor structure, but may still satisfy the above stability conditions. These latter cases are particularly relevant in the present context, since cross-diffusion may induce instability of an equilibrium that would otherwise remain stable under standard diffusion.

The general form of the characteristic matrix is 
\begin{equation*}
M_{n}(u^*, v^*)\coloneqq 
\begin{pmatrix}
J_{11}-\pa_1D\lambda_n  & J_{12}-\pa_2D\lambda_n\\
J_{21} & J_{22}-d_v\lambda_n
\end{pmatrix},
\end{equation*}
where $(J_{ij})_{i,j=1,2}$ are the elements of the Jacobian matrix evaluated at $(u^*,v^*)$, and $\lambda_n,\, n\in\N$ are the eigenvalues of the Laplace operator in the Neumann case.
In all cases, the characteristic matrix has a negative trace, and thus, Turing instability can occur only if 
\begin{equation}\label{det coex equi <0 aggregation}
    \det M_{n}(u^*, v^*)=d_v\pa_1D(u^*, v^*)\lambda_n^2+B\lambda_n+\det J^*<0,
\end{equation}
with 
\begin{equation}\label{def B coex equi aggregation}
    B=-\pa_1D(u^*, v^*)J_{22}-d_vJ_{11}+\pa_2D(u^*, v^*)J_{21}.
\end{equation}
Thus, it is necessary to have $B<0$.

In the following, we systematically analyse these sign structures and identify the conditions on the diffusivity function $D$ under which cross-diffusion-driven instability may occur. Moreover, we link this property to the associated fast-reaction system having the cross-diffusion system~\eqref{generalXDmodel} as its fast-reaction limit, here written in the general form
\begin{equation}\label{mesogeneral}
\begin{cases}
\pa_t u_a-d_a\Delta u_a=f_1(u_a,u_b,v)+\f 1{\e}\left[k(v)u_b-h(v)u_a\right],\\[0.2cm]
\pa_t u_b-d_b\Delta u_b=f_2(u_a,u_b,v)-\f 1{\e}\left[k(v)u_b-h(v)u_a\right],\\
\pa_t v-d_v\Delta v=\tilde g(u_a,u_b, v),
\end{cases}
\end{equation}
where the reaction part is expressed in terms of the functions $f_1,\, f_2, \, \tilde g$. The transition rate $h(v)$ is nondecreasing, whereas $k(v)$ is nonincreasing.

\subsection{Non-activator-inhibitor structure}
We focus here on the four sign structures that do not correspond to the activator–inhibitor mechanism. In these cases, standard diffusion alone cannot destabilise the homogeneous equilibrium, whereas cross-diffusion may induce the emergence of spatial patterns.

Since we consider triangular cross-diffusion acting only on the first species, the sign of $J_{12}=f_v$ does not influence the diffusion-driven instability mechanism. Consequently, the four admissible non-activator--inhibitor structures can be grouped in pairs according to the sign of $J_{21}=g_u$ and of the diagonal entries.

In the first case, we have
\[
J^* =
\begin{pmatrix}
- & \pm \\
- & -
\end{pmatrix}.
\]
Under the condition on the coefficient $B$ in equation~\eqref{def B coex equi aggregation}, we deduce that the diffusivity function $D$ must be increasing with respect to $v$ for cross-diffusion-driven instability to occur. Regarding the associated fast-reaction system~\eqref{mesogeneral}, it means that species $u_a$ must have a higher diffusion coefficient than species $u_b$.

\begin{remark}
This is the case of the competitive reaction part, in which the only cross-diffusion able to destabilise the homogeneous equilibrium is that of the SKT model (namely, $D(u,v)=(d_u+d_{12}\phi(v))u$ with $\phi$ increasing).
\end{remark}

In the second case, we have
\[
J^* =
\begin{pmatrix}
- & \pm \\
+ & -
\end{pmatrix}.
\]
Under the condition on the cross-diffusion term $B$ in equation~\eqref{def B coex equi aggregation}, we deduce that the diffusivity function $D$ must be decreasing with respect to $v$ for cross-diffusion-driven instability to occur. Regarding the associated fast-reaction system~\eqref{mesogeneral}, it means that species $u_a$ must have a smaller diffusion coefficient than species $u_b$.

\begin{remark}
This is the case of the biomass-autotoxicity model~\cite{GianninoIuorioSoresina} and of the metal surface corrosion model~\cite{Corrosione}, in which the only cross-diffusion able to destabilise the homogeneous equilibrium describes a propagation reduction of roots due to autotoxicity, namely $D(u,v)=(d_u-d_{12}\phi(v))u$ with~$\phi$ an increasing function.
\end{remark}

\subsection{Activator-inhibitor structure}
We focus here on the four sign structures that correspond to the activator–inhibitor mechanism. In these cases, standard diffusion alone can destabilise the homogeneous equilibrium, and therefore cross-diffusion may either enhance or prevent the emergence of spatial patterns. As in the previous subsection, since we consider triangular cross-diffusion acting only on the first species, the four admissible activator--inhibitor structures can be grouped in pairs.

In the first case, we have
\[
J^* =
\begin{pmatrix}
- & - \\
+ & +
\end{pmatrix},\qquad \textnormal{or} \qquad 
J^* =
\begin{pmatrix}
+ & - \\
+ & -
\end{pmatrix}.
\]
Under the condition on the cross-diffusion term $B$ in equation~\eqref{def B coex equi aggregation}, we deduce that a diffusivity function $D$ increasing with respect to $v$ (corresponding to $d_b>d_a$ in the associated fast-reaction system~\eqref{mesogeneral}) reduces the instability region, while a diffusivity function $D$ decreasing with respect to $v$ (corresponding to $d_b<d_a$ in the associated fast-reaction system~\eqref{mesogeneral}) enhances the possibility of instability.

In the second case, we have
\[
J^* =
\begin{pmatrix}
+ & + \\
- & -
\end{pmatrix},\qquad \textnormal{or} \qquad 
J^* =
\begin{pmatrix}
- & + \\
- & +
\end{pmatrix}.
\]
Under the condition on the cross-diffusion term $B$ in equation~\eqref{def B coex equi aggregation}, we deduce that a diffusivity function $D$ decreasing with respect to $v$ (corresponding to $d_b<d_a$ in the associated fast-reaction system~\eqref{mesogeneral}) reduces the instability region, while a diffusivity function $D$ increasing with respect to $v$ (corresponding to $d_b>d_a$ in the associated fast-reaction system~\eqref{mesogeneral}) enhances the possibility of instability.

\begin{remark}
This is the case of the predator--prey model in~\cite{DesvillettesSoresina}.
\end{remark}


\section{Dietary diversity and starvation}
In this section, we consider a class of triangular cross-diffusion systems that describe the competition between two species in the context of dietary diversity and starvation (DDS). We briefly recall its derivation via a fast-reaction limit. We then discuss the Turing instability in relation to the analytical structure of the limit system. 

We consider the following system,
\begin{equation}\label{DDS system}
\begin{cases}	
\partial_{t}u-\Delta(d_{a}u_a+d_{b}u_b)=u(r_{u}-r_{11}u-r_{12}v),\hspace{1cm}&(t,x)\in\mathbb{R}_{+}\times\,\Omega,\\
\partial_{t}v-d_{v}\Delta v=v(r_{v}-r_{21}u-r_{22}v),
&(t,x)\in\mathbb{R}_{+}\times\,\Omega,\\
\end{cases}
\end{equation}
with homogeneous zero flux boundary conditions. The diffusion coefficients are $d_a,\, d_b$, and the reaction terms are of Lotka-Volterra competitive interactions type, as considered in both SKT systems~\eqref{SKT model +},~\eqref{SKT model -}, where the functions $u_a=u_a(u,v)$ and $u_b=u_b(u,v)$ represent two sub-populations of the density $u$. The pair $(u_a,u_b)$ is the (unique) solution of 
\begin{equation}\label{NL system DDS}
\begin{cases}
    u=u_a+u_b,\\
    Q_{\DDS}(u_a,u_b,v)=k(bu_b+dv)u_b-h(au_a+cv)u_a=0,
\end{cases}
\end{equation}
with parameters $a,\,b,\,d>0$ and $c\ge 0$, and where the transition functions $h$ and~$k$ in~\eqref{NL system DDS} belong to~$C^1([0,+\infty))$ and are both increasing. Specifically, there exist $C_h, \,C_k, \,C_{h'}, \,C_{k'}>0$ such that for all $w\ge0$ (see \cite{BCDK})
\begin{equation}\label{hp h,k in DDS}
h(w)\ge C_h>0, \quad k(w)\ge C_k>0, \quad h'(w)\le C_{h'},\quad k'(w)\le C_{k'}.
\end{equation}
We denote $Q_\DDS= Q_{\DDS}(u_a,u_b,v)$ as in \eqref{NL system DDS} and we define
\begin{equation}\label{def A DDS}
    D_\DDS(u,v)\coloneqq d_a u_a(u,v)+d_b u_b(u,v),
\end{equation}
with $(u_a,u_b)$ satisfying \eqref{NL system DDS}, \eqref{hp h,k in DDS}. 

The system  \eqref{DDS system}, \eqref{NL system DDS} can be rigorously obtained by the fast-reaction limit as $\e\to 0$ of the following fast-reaction system (see \cite{BCDK, Brocchieri2025})
\begin{equation}\label{fast system DDS}
\begin{cases}
\pa_t u_a-d_a\Delta u_a=u_a(r_{u}-r_{11}u-r_{12}v)+\f 1{\e}\left[k(bu_b+dv)u_b-h(au_a+cv)u_a\right],\\[1ex]
\pa_t u_b-d_b\Delta u_b=u_b(r_{u}-r_{11}u-r_{12}v)-\f 1{\e}\left[k(bu_b+dv)u_b-h(au_a+cv)u_a\right],\\[1ex]
\pa_t v-d_v\Delta v=v(r_{v}-r_{21}u-r_{22}v),
\end{cases}
\end{equation}
with the diffusion coefficients $d_a,\,d_b,\,d_v>0$. The system \eqref{fast system DDS} models the effect of diet diversity and starvation of species $u,\,v$ in a competitive interaction context and drives the emergence of cross-diffusion terms in \eqref{DDS system} and \eqref{NL system DDS} at the limit $\e\to 0$. In fact, the fast-reaction term $Q_\DDS$ describes the switching between the sub-species $u_a$ and $u_b$ depending on food availability. More specifically, individuals of sub-species $u_a$ (resp.~$u_b$) can freely change the type of food with a transition rate $h$ (resp.~$k$), depending on the amount $au_a+cv$ (resp.~$bu_b+dv$) that competes for the same resource.

Regarding the analysis for the Turing instability, we refer to Proposition~\ref{ode stability prop} for the local stability analysis of the associated ODE system~\eqref{DDS system} and~\eqref{NL system DDS}. To linearise the diffusion part, we first differentiate a general identity 
\[
Q\left(u_a(u,v), u_b(u,v), v\right)=0,
\] 
with respect to $u$ and $v$ and we get, using $u=u_a+u_b,$  (see \cite[Section 3.8.2]{BrocchieriPhD})
\[
(1-\pa_1u_b)\,\pa_1Q+\pa_1u_b\,\pa_2Q=0,\qquad 
-\pa_2u_b\,\pa_1Q+\pa_2u_b\,\pa_2Q+\pa_3Q=0,
\]
that implies
\begin{equation}\label{def grad ub}
\pa_1u_b=-\f{\pa_1Q}{\pa_2Q-\pa_1Q}=1-\pa_1 u_a,    \qquad \pa_2u_b=-\f{\pa_3Q}{\pa_2Q-\pa_1Q}=-\pa_2 u_a.
\end{equation}
Specifically, in the DDS case, $Q_\DDS$ defined in \eqref{NL system DDS} and~\eqref{hp h,k in DDS} satisfies the following properties
\begin{equation}\label{def grad QDDS}
\begin{split}
\pa_1 Q_\DDS&= -(h(au_a+cv)+au_ah'(au_a+cv))<0,\\
\pa_2 Q_\DDS&=k(bu_b+dv)+bu_bk'(bu_b+dv)>0,\\
\pa_3 Q_\DDS&=du_bk'(bu_b+dv)-cu_ah'(au_a+cv),
\end{split}
\end{equation}
giving $\pa_2 Q_\DDS-\pa_1 Q_\DDS>0$. Therefore, the quantities in \eqref{def grad ub} are such that
\begin{equation}\label{pa1 ua, ub properties}
    \pa_1 u_a\in (0,1)\qquad\text{and}\qquad \pa_1 u_b\in (0,1),
\end{equation}
and using \eqref{def grad QDDS} 
\begin{equation}\label{pa2 ua, ub properties}
-\f ca\le \pa_2 u_a\le \f db\qquad\text{and}\qquad-\f db\le \pa_2 u_b\le \f ca.
\end{equation}
We now consider the diffusivity function $D_\DDS$ in \eqref{def A DDS}, and we compute
\begin{equation}\label{grad A DDS}
\begin{split}
    \pa_1D_\DDS(u,v)&=d_a + (d_b-d_a)\pa_1 u_b(u,v),\\
    \pa_2 D_\DDS (u,v)&= (d_b-d_a) \pa_2 u_b(u,v),
\end{split}
\end{equation}
that satisfy by \eqref{pa1 ua, ub properties} and \eqref{pa2 ua, ub properties}
\begin{equation}\label{bounds grad D DDS}
\begin{split}
    0<&\min\{d_a,d_b\}\le \pa_1 D_\DDS(u,v)\le \max\{d_a,d_b\},\\
    &|\pa_2 D_{\DDS}(u,v)|\le (d_a+d_b)\left(\f ca+\f db\right).
\end{split}
\end{equation}
By linearising the system \eqref{DDS system} and~\eqref{NL system DDS} around the coexistence equilibrium $(u^*,v^*)$ in~\eqref{homo equi coexistence}, we observe that \eqref{def matrix Mn+}--\eqref{def B coex equi} hold by replacing the diffusivity function with $D_\DDS$. From \eqref{def B coex equi} and due to the strict positivity of $\pa_1D_\DDS$ in \eqref{bounds grad D DDS}, a necessary condition for Turing instability is $\pa_2 D_\DDS>0$, i.e., by \eqref{def grad ub}, \eqref{def grad QDDS}, and~\eqref{def A DDS}
\begin{equation}\label{necessary cond Turing DDS}
    (d_b-d_a)\pa_3 Q_\DDS<0.
\end{equation}

By choosing the affine transition functions
\[
    h(w)=A+w,\qquad k(w)=B+w,\qquad w\ge0,
\]
with $A,B>0$ the condition \eqref{necessary cond Turing DDS} becomes
\begin{equation}\label{DDS_nec_cond_Turing specific}
(d_b-d_a)\left(du_b(u^*,v^*)-cu_a(u^*,v^*)\right)<0.
\end{equation}
If we consider $c=0$ as in \cite{BCDK}, then a necessary condition for Turing instability is $d_b<d_a$, while the case $c>0$ as in~\cite{Brocchieri2025} drives the analysis to different scenarios, since the sign of $\pa_3Q_{\DDS}$ is not immediately determined. 

\begin{remark}\label{rmk h,k power law DDS}
We observe that in \cite{Brocchieri2025} the authors derived the cross-diffusion system~\eqref{DDS system} and \eqref{NL system DDS} by the fast-reaction limit, extending the hypothesis of the transition functions in \eqref{hp h,k in DDS} and considering $h,\,k$ as power law functions
\begin{equation}\label{def h,k in DDS}
    h(w)=(A+w)^\alpha,\qquad k(w)=(B+w)^\beta,\qquad w\ge0,
\end{equation}
with $A>0,\,B\ge0$ and suitable $0<\alpha\le\beta$. The assumption in \eqref{def h,k in DDS} still satisfies the properties of the cross-diffusion function $D_\DDS$ in \eqref{grad A DDS}. 
\end{remark}
\section{Conclusions}

In this work, we have investigated the role of cross-diffusion terms arising from fast-reaction limits in the onset of diffusion-driven instability. Our analysis highlights how the structure of the cross-diffusion term, and in particular its dependence on the underlying fast-reaction mechanism, plays a decisive role in determining whether spatial patterns may emerge.

In the framework of SKT-type competition models, we have shown that cross-diffusion terms derived from avoidance behaviour can induce Turing instability under suitable conditions, even in the absence of a classical activator--inhibitor structure. In contrast, alternative fast-reaction formulations that decrease diffusivity (hiding behaviour) prevent the occurrence of diffusion-driven instability, despite having the same reaction dynamics and a cross-diffusion term.

These results emphasise that not all biologically consistent cross-diffusion mechanisms are equivalent from the point of view of pattern formation. The qualitative behaviour of the system depends critically on the sign and structure of the cross-diffusion terms, which in turn reflect the underlying biological assumptions.

Finally, the classification based on the sign structure of the reaction Jacobian provides a unified framework for understanding when cross-diffusion can promote or inhibit instability. This perspective clarifies the interplay between reaction kinetics and diffusion structure, and may serve as a guideline for the modelling and analysis of multi-species systems with density-dependent dispersal. In particular, our results show that the ability of cross-diffusion to trigger pattern formation is not a generic feature, but rather a structural property tied to its fast-reaction origin.

Our analysis does not cover all possible cross-diffusion structures arising from fast-reaction mechanisms. In particular, models such as those based on cockroach aggregation or dietary diversity and starvation lead to more general diffusion terms that fall outside the present structural classification. 

We conclude this section with a comparison between the SKT models~\eqref{SKT model +} and~\eqref{SKT model -} and the DDS model \eqref{DDS system} and \eqref{NL system DDS}, in which we outline some relevant structural differences that we believe are crucial for understanding the fast-reaction mechanisms and pattern formation. We first remark that in the DDS framework, the diffusivity function $D_{\DDS}$ is such that $\pa_2D_{\DDS}$ is bounded both from below and from above (see \eqref{grad A DDS}), while this does not hold in the SKT framework, neither for $\pa_2D_+$ nor for $\pa_2D_-$ (see \eqref{grad A+} and \eqref{grad A-}, respectively). The density $u$ does not a priori belong to $L^{\infty}(\R_+\times\Omega)$, and more generally, the boundedness of $u$ may represent a challenging issue due to the nonlinear structure of the cross-diffusion system. Additionally, it is worth mentioning that in the SKT framework, the sign of $\pa_2D_+$ and $\pa_2D_-$ is directly determined by the monotonicity of the cross-diffusion function $\phi$, while this is not the case in the DDS framework. This remark is a direct consequence of two fundamental properties: first, the SKT framework requires an assumption on the relative size of the diffusion coefficients $d_a$ and $d_b$, which is not needed in the DDS framework. Secondly, the fast-reaction part in the SKT framework, $Q_{\SKT}$,
is such that $\pa_3Q_{\SKT}<0$, due to the monotonicity of the transition rates. However, the sign of $\pa_3Q_{\DDS}$ is not directly determined (see \eqref{def grad QDDS}), and consequently, the condition for Turing instability \eqref{necessary cond Turing DDS} may give rise to different scenarios.

Finally, we believe that the dependence of the subpopulations $u_a$ and $u_b$ on the respective transition rates $h$ and $k$ in $Q_{\DDS}$ is not constrained to a modelling choice but also yields interesting results in the mathematical analysis, including pattern formation. Extending the analysis to this broader class represents an interesting direction for future work, although it introduces additional technical challenges.


\bigskip

\bigskip

\begin{acknowledgment}
The authors thank Romina Travaglini for the helpful discussion on the topic of the paper. The results presented in this paper were partially presented at WASCOM 2025. EB is a member of the INdAM-GNAMPA Group. C.S.~is a member of the INdAM-GNFM Group and has received funding from the Programma Giovani Ricercatori ``Rita Levi Montalcini'' 2021.
\end{acknowledgment}

\bigskip

\noindent
On behalf of all authors, the corresponding author states that there is no conflict of
interest.


\begin{thebibliography}{99}

\bibitem{BCDK}
E.~Brocchieri, L.~Corrias, H.~Dietert, and Y.-J.~Kim.
{\em Evolution of dietary diversity and a starvation-driven cross-diffusion system as its singular limit}, 
Journal of Mathematical Biology 83(5):58 (2021).

\bibitem{Brocchieri2025}
E.~Brocchieri, and L.~Corrias.
{\em On a class of triangular cross-diffusion systems and its fast reaction approximation}, arXiv preprint  arXiv:2503.07156, 2025.

\bibitem{BrocchieriPhD}
E.~Brocchieri.
{\em Evolutionary dynamics of populations structured by dietary diversity and starvation: cross-diffusion systems}, PhD thesis, 2023.

\bibitem{BoutonDesvillettesDietert2025}
H.~Bouton, L.~Desvillettes, and H.~Dietert.
{\em Global strong solutions for the triangular Shigesada-Kawasaki-Teramoto cross-diffusion system in three dimensions and parabolic regularisation for increasing functions}. 
arXiv preprint arXiv:2503.08186, 2025.

\bibitem{Breden2019}
M.~Breden, C.~Kuehn, and C.~Soresina.
{\em On the influence of cross-diffusion in pattern formation}, 
Journal of Computational Dynamics, 8(2):213--240, 2021.

\bibitem{BrocchieriDesvillettes}
E.~Brocchieri, L.~Desvillettes, and H.~Dietert.
{\em Study of a class of triangular starvation driven cross-diffusion systems}, 
Ricerche di Matematica, 74(3):1373--1399, 2025.

\bibitem{ConfortoDesSoresina2018}
F.~Conforto, L.~Desvillettes, and C.~Soresina.
{\em About reaction-diffusion systems involving the Holling-type II and the Beddington--De Angelis functional responses for predator prey models},
Nonlinear Differential Equations and Applications, 25:24, 2018.

\bibitem{Corrosione}
G.~Consolo, G.~Inferrera, E.~Proverbio, and C.~Soresina.
{\em An extended corrosive-passivating model with cross-diffusion for the initiation of corrosion patterns}. Physica D: Nonlinear Phenomena, 483:134986, 2025.

\bibitem{DesvillettesFiorentinoMautone}
L.~Desvillettes, L.~Fiorentino, and T.~Mautone.
\emph{Fast reaction limit for a Leslie–Gower model including preys, meso-predators and top-predators}. Nonlinear Analysis, 258:113817, 2025.

\bibitem{Desvillettes2015}
L.~Desvillettes, and A.~Trescases.
{\em New results for triangular reaction cross diffusion system},
Journal of Mathematical Analysis and Applications, 430(1):32--59, 2015.

\bibitem{DesvillettesRDM}
L.~Desvillettes.
{\em About the triangular Shigesada–Kawasaki–Teramoto reaction cross diffusion system}. 
Ricerche di Matematica, 73(1):105--114, 2024.

\bibitem{DesvillettesSoresina}
L.~Desvillettes, and C.~Soresina.
{\em Non-triangular cross-diffusion systems with predator–prey reaction terms}. 
Ricerche di Matematica, 68(1):295--314, 2019.

\bibitem{cockroaches}
J.~Eliaš, H.~Izuhara, M.~Mimura, and B.Q.~Tang.
{\em An aggregation model of cockroaches with fast-or-slow motion dichotomy}. Journal of Mathematical Biology, 85(3):28, 2022.

\bibitem{Elias2018}
J.~Eliaš, M.H.~Kabir, and M.~Mimura.
{\em On the well-posedness of a dispersal model for farmers and hunter–gatherers in the Neolithic transition}. Mathematical Models and Methods in Applied Sciences, 28(2):195--222, 2018.

\bibitem{GambinoIGP}
F.~Farivar, G.~Gambino, V.~Giunta, M.C.~Lombardo, and M.~Sammartino. \emph{Intraguild predation communities with anti-predator behavior}. SIAM Journal on Applied Dynamical Systems, 24(2):1110--1149, 2025.

\bibitem{Gambino2012}
G.~Gambino, M.C.~Lombardo, and M.~Sammartino.
{\em Turing instability and traveling fronts for a nonlinear reaction–diffusion system with cross-diffusion}. Mathematics and Computers in Simulation, 82(6):1112--1132, 2012.

\bibitem{Gambino2013}
G.~Gambino, M.C.~Lombardo, and M.~Sammartino.
{\em Pattern formation driven by cross-diffusion in a 2D domain.} 
Nonlinear Analysis: Real World Applications, 14(3):1755--1779, 2013.

\bibitem{GianninoIuorioSoresina}
F.~Giannino, A.~Iuorio, and C.~Soresina.
{\em Beyond water limitation in vegetation-autotoxicity patterning: a cross-diffusion model.} 
arXiv preprint arXiv:2506.03981, 2025.

\bibitem{IIR}
M.~Iida, H.~Izuhara, and R.~Kon. 
\emph{Cross-diffusion predator–prey model derived from the dichotomy between two behavioral predator states}. 
Discrete and Continuous Dynamical Systems-B, 28(12):6159--6178, 2023

\bibitem{IMN}
M.~Iida, M.~Mimura, and H.~Ninomiya,
{\em Diffusion, Cross-diffusion and Competitive Interaction},
Journal of Mathematical Biology 53:617--641, 2006.

\bibitem{Morgan2026}
J.~Morgan, C.~Soresina, B.Q.~Tan, and B.-N.~Tran.
{\em Singular limit and convergence rate via projection method in a model for plant-growth dynamics with autotoxicity}.
Journal of Differential Equations, 452:113797, 2026.

\bibitem{SKT}
N.~Shigesada, K.~Kawasaki, and E.~Teramoto.
{\em Spatial segregation of interacting species}. 
Journal of Theoretical Biology, 79(1):83--99, 1979.

\bibitem{BaoNgocMM}
B.Q.~Tang, and B.-N.~Tran.
{\em Rigorous derivation of Michaelis–Menten kinetics in the presence of slow diffusion}. 
SIAM Journal on Mathematical Analysis, 56(5):5995--6024, 2024.

\bibitem{villar2025designing}
E.~Villar-Sep{\'u}lveda, A.R.~Champneys, and A.L.~Krause.
{\em Designing reaction-cross-diffusion systems with Turing and wave instabilities}, 
Journal of Mathematical Biology, 91(4):37, 2025.

\end{thebibliography}
\end{document}